\documentclass[11pt,letterpaper]{amsart}

\usepackage[english]{babel}
\usepackage[OT2,T1]{fontenc} 

\usepackage{mathtools}
\usepackage{amssymb}
\usepackage{stmaryrd}
\usepackage{bm}
\usepackage{tikz-cd}

\usepackage{comment}
\usepackage{todonotes}
\usepackage{xcolor}
\usepackage{enumitem}
\usepackage{graphicx}
\usepackage{setspace}
\usepackage{url}
\usepackage{xr}

\usepackage[title]{appendix}
\usepackage{chngcntr}
\usepackage{apptools}
\usepackage{thmtools}

\usepackage[pagebackref]{hyperref}
\usepackage[nameinlink]{cleveref}
\usepackage{bookmark}

\definecolor{crimsonglory}{rgb}{0.75, 0.0, 0.2}
\definecolor{darkpowderblue}{rgb}{0.0, 0.2, 0.6}

\hypersetup{
	colorlinks = true,
	urlcolor   = darkpowderblue,
	linkcolor  = darkpowderblue,
	citecolor  = crimsonglory
}

\theoremstyle{plain}
\newtheorem{theorem}{Theorem}[section]
\newtheorem{definition}[theorem]{Definition}

\newtheorem{prop}[theorem]{Proposition}

\newtheorem{lemma}[theorem]{Lemma}

\newtheorem{remark}[theorem]{Remark}
\newtheorem{remarks}[theorem]{Remarks}

\DeclareMathOperator{\homm}{Hom}

\DeclareMathOperator{\End}{End}
\DeclareMathOperator{\GL}{GL}
\DeclareMathOperator{\SL}{SL}

\DeclareMathOperator{\diag}{diag}

\DeclareMathOperator{\tr}{tr}

\DeclareMathOperator{\spec}{Spec}

\DeclareMathOperator{\disc}{disc}
\DeclareMathOperator{\zar}{Zar}
\DeclareMathOperator{\chara}{char}
\DeclareMathOperator{\crys}{crys}

\DeclareMathOperator{\ord}{ord}

\DeclareMathOperator{\arch}{arch}

\DeclareMathOperator{\ssing}{ssing}

\newcommand{\Z}{\mathbb{Z}}

\renewcommand{\C}{\mathbb{C}}
\newcommand{\Q}{\mathbb{Q}}

\newcommand{\A}{\mathbb{A}}

\newcommand{\F}{\mathbb{F}}

\newcommand{\CE}{\mathcal{E}}

\newcommand{\CP}{\mathcal{P}}
\newcommand{\CO}{\mathcal{O}}

\DeclareSymbolFont{cyrletters}{OT2}{wncyr}{m}{n}
\DeclareMathSymbol{\Sha}{\mathalpha}{cyrletters}{"58}

\newcommand{\mathsym}[1]{{}}
\newcommand{\unicode}[1]{{}}

\title{Local values of G-functions at CM points}
\author{Georgios Papas}
\address{Faculty of Mathematics and Computer Science\\
	The Weizmann Institute of Science\\
	234 Herzl Street\\
	Rehovot 76100\\
	Israel, and\\
	Institute for Advanced Study\\
	1 Einstein Drive\\
	Princeton, N.J. 08540\\
	U.S.A.}

\email{georgios.papas@weizmann.ac.il}
\email{gpapas@ias.edu}
\begin{document}
	\maketitle
	
	\begin{abstract}
	We study G-functions attached to one--parameter families of elliptic curves centered at a CM point. Using de Rham--Betti and de Rham--crystalline period matrices, we construct non-trivial algebraic relations among their values at CM points at every place of proximity, extending work of Andr\'e and Beukers. These relations yield effective height bounds via the Andr\'e--Bombieri Hasse principle. We also discuss a connection with class number bounds for imaginary quadratic fields.
	\end{abstract}
	
	\section{Introduction}
 
Let $f:\CE\rightarrow S$ be a one--parameter family of elliptic curves defined over a number field $K$, and let $s_0\in S(K)$ be a point such that $\CE_{s_0}$ has complex multiplication. Associated with $f$ and centered at $s_0$ is a family $Y_G$ of G-functions. Our goal is to construct algebraic relations among their values at other points $s\in S(\bar{\Q})$ for which $\CE_s$ also has complex multiplication.
 
After a series of finite modifications required in the construction of the G-functions, one obtains a finite cover $c:S'\rightarrow S$, a local parameter $x$ at the points lying above $s_0$, and a collection of analytic discs centered at these points. Given a CM point $s\in S(\bar{\Q})$ and a place $v$ of a number field $L\supseteq K(s)$, we say that $s$ is $v$-adically close to $s_0$ if some lift of $s$ to $S'$ belongs to one of these discs, and call such a $v$ a place of proximity of $s$. The precise construction is recalled in \Cref{section:backgroundgfuns}.
 
 Our main result gives a non-trivial algebraic relation among the corresponding G-function values at every place of proximity. For a number field $L$ and a place $v\in\Sigma_L$, write $\iota_v:L\hookrightarrow\C_v$ for the corresponding embedding, and let $I(\SL_2)$ denote the ideal defining $\SL_2$ as a subvariety of $M_{2\times2}\simeq\mathbb{A}^4$.
 
\begin{theorem}\label{naiverelations}
	Let $f:\CE\rightarrow S$, $s_0$, and $Y_G$ be as above. There exists an absolute constant $d>0$ with the following property.
	
	Let $s\in S(\bar{\Q})$ be such that $\CE_s$ has complex multiplication. Then there exists a number field $L_s\supseteq K(s)$ with $[L_s:K(s)]\leq d$ such that, for every place $v\in\Sigma_{L_s}$ of proximity of $s$, there exists a polynomial $R_{s,v}\in L_s[X_{i,j}:1\leq i,j\leq2]$ of degree at most $2$ satisfying
	$$
	\iota_v\bigl(R_{s,v}(Y_G(x(s)))\bigr)=0
	$$
	and $R_{s,v}\notin I(\SL_2)$.
	
	Moreover, at places of proximity at which $\CE_{s_0}$ has potentially ordinary reduction one may choose a single polynomial $R_{s,\mathrm{ord}}$ independent of $v$.
\end{theorem}
 
The problem of describing such local relations was considered by Andr\'e \cite{andremots} and Beukers \cite{beukers}, who used different methods to obtain partial answers. In \cite{beukers}, Beukers works under several integrality assumptions on the family $f:\CE\rightarrow S$ and the fiber $\CE_{s_0}$, and obtains relations at both archimedean and finite places, the latter subject to additional local and isogeny hypotheses and excluding places above $2$ and $3$. At archimedean places his arguments use the de Rham--Betti comparison. In contrast, at finite places he lifts isogenies of the reductions to isogenies of families in characteristic zero.

In \cite{andremots}, Andr\'e gives relations at places $v$ at which $\CE_{s_0}$ has supersingular reduction, in the case where $\CE_s$ and $\CE_{s_0}$ have distinct CM fields. Andr\'e's arguments, in contrast to Beukers's finite-place construction, are cohomological in nature and use the de Rham--crystalline comparison isomorphism in place of the characteristic-zero de Rham--Betti comparison. A key ingredient in Andr\'e's approach is the use of $p$-adic Betti lattices; see also \cite{andremsj} and \Cref{section:supersingbases} below for a summary.

Our approach is purely cohomological and places the constructions of Andr\'e and Beukers, together with the remaining local cases, in a single framework based on de Rham--Betti and de Rham--crystalline period matrices. In particular, it recovers the finite-place relations of Beukers that are not covered by Andr\'e's construction, without Beukers's additional assumptions on the family $f:\CE\rightarrow S$ or restrictions on the place $v$.

A new ingredient in the ordinary case is the use of Frobenius to choose crystalline bases adapted to the slope decomposition. This allows both the periods of the central CM fiber and the $v$-dependent scalars arising from the crystalline action of the isogeny to be eliminated entirely; see \Cref{section: outline of relation construction intro} below.

The setting considered here, with the G-functions centered at a CM point, differs from the case where $s_0$ corresponds to a degeneration, that is, to a boundary point of the ambient moduli space. In that setting, $p$-adic proximity of CM points to $s_0$ is ruled out by an argument based on Gabber's lemma; see \cite[Ch.~X, Lemma~3.4]{andre1989g} and \cite[Lemma~5.1]{papaseffbrsieg}. Thus only archimedean proximity can occur. At an archimedean place, however, the analogous strategy based on the comparison isomorphism does not eliminate all transcendental constants and thus does not lead to a non-trivial relation with algebraic coefficients among the entries of the G-functions; see \cite[Remark~3, p.~201]{andre1989g}.

\subsection{Outline of the relation construction}\label{section: outline of relation construction intro}
Over each of the analytic discs centered above $s_0$, the relative de Rham--Betti or de Rham--crystalline comparison identifies the cohomology of the central fiber with the space of horizontal sections of the relative de Rham cohomology. Comparing the resulting period matrices at $s_0$ and at a point $s$ in the same disc gives
$$
\mathcal{P}_v(s)=\iota_v(Y_G(x(s)))\mathcal{P}_v(s_0).
$$

To construct relations among the entries of $\iota_v(Y_G(x(s)))$, we use the additional cohomological cycles provided by the CM structure of $\CE_s$ to obtain relations among the entries of $\mathcal{P}_v(s)$, and then eliminate the possibly transcendental entries of $\mathcal{P}_v(s_0)$ through the comparison identity above. This strategy underlies both Beukers's construction of the archimedean relations and Andr\'e's construction of the supersingular relations described above.

For finite places, we distinguish cases according to the reduction type of the central fiber $\CE_{s_0}$ at $v$. In each case, we choose a basis of the crystalline cohomology of the reduction of $\CE_{s_0}$ in which these cycles take a controlled form. CM elliptic curves have potentially everywhere good reduction; see \cite{serretate}. Thus, after replacing $K$ by a finite extension, we assume throughout this discussion that $\CE_{s_0}$ has everywhere good reduction.

\subsubsection{Ordinary places}
Let $v$ be a place of ordinary reduction of $\CE_{s_0}$. In this case, the Frobenius action on the crystalline cohomology of $\widetilde{\CE}_{s_0,v}$ has distinct slopes. We choose a basis adapted to the resulting slope decomposition; see \Cref{section:ordredbases}.

The ordinary relation is given by the vanishing of a fixed matrix entry, independently of $v$. The relation is constructed in \Cref{section:cmordpadic}. We note that \Cref{thmordrels} recovers \cite[Theorem~5.4]{beukers} without Beukers's restrictions on either the family or the residue characteristic of $v$; see also \Cref{remark ordinary recover beukers}.

\subsubsection{Supersingular places}
Let $v$ be a place of supersingular reduction of $\CE_{s_0}$. The Frobenius action on the crystalline cohomology of $\widetilde{\CE}_{s_0,v}$ no longer has distinct slopes, so the slope decomposition used in the ordinary case is unavailable. Instead, we use Andr\'e's $p$-adic Betti lattices to choose a distinguished basis; see \Cref{section:supersingbases}.

Andr\'e's construction in \cite{andremots} gives the required relations when $\CE_s$ and $\CE_{s_0}$ have distinct CM fields. It remains to treat the case in which their CM fields coincide; the corresponding relations are constructed in \Cref{section:samecmfield}.	
	\subsection{Height bounds and applications}\label{section:heightboundsintro}

Let $h$ denote a Weil height on $S$. The Andr\'e--Bombieri Hasse principle gives an effective upper bound for $h(s)$ that is polynomial in the degree of a global relation obtained by multiplying the local relations constructed in \Cref{naiverelations}. The contribution of the archimedean relations is bounded in terms of $[K(s):\Q]$. Thus the only contribution not already controlled by $[K(s):\Q]$ comes from places of proximity at which $\CE_{s_0}$ has potentially supersingular reduction.

For a CM point $s\in S(\bar{\Q})$ distinct from $s_0$, let $\CP(s,s_0)\subset\Sigma_{K(s)}$ denote the set of finite places of proximity of $s$ to $s_0$ at which $\CE_{s_0}$ has potentially supersingular reduction.

The following is an analogue, in the present setting, of the height bound announced by Andr\'e following \cite[Theorem~1]{andremots}.

\begin{theorem}\label{htboundgfunsgeneral}
	Let $f:\CE\rightarrow S$ be a one--parameter family of elliptic curves defined over a number field $K$, and let $s_0\in S(K)$ be such that $\CE_{s_0}$ has CM by the field $F_0$. Assume that the induced map $S\rightarrow Y(1)$ is non-constant.
	
	There exist effectively computable constants $C_1,C_2>0$, depending only on $f:\CE\rightarrow S$ and $s_0$, such that
	\begin{equation}
		h(s)\leq C_1\big([K(s):K]+|\CP(s,s_0)|\big)^{C_2}
	\end{equation}
	for every $s\in S(\bar{\Q})$ such that $\CE_s$ has CM by a field distinct from $F_0$. 
\end{theorem}

If $\CE_s$ and $\CE_{s_0}$ have the same CM field, a stronger effective logarithmic height bound follows by combining the isogeny estimates of Masser--W\"ustholz \cite{masserwustholzisogenies} with work of Pazuki \cite{pazuki}; see \Cref{propisogenouscm}.

\subsubsection{Supersingular proximity and class number bounds}

For the $j$-family, the height bound in \Cref{htboundgfunsgeneral} is related to the effective Brauer--Siegel problem. We discuss this in \Cref{section:brauersiegel}, where we show that a sufficiently strong bound for $|\CP(s,s_0)|$ would imply an effective Landau-type lower bound for the class numbers of imaginary quadratic fields. We conclude by relating the set $\CP(s,s_0)$ to work of Lauter--Viray \cite{lauterviray}, extending work of Gross--Zagier \cite{grosszagier}.

\subsection{Notation}
Let $E$ be an elliptic curve over a number field $K$ and $v\in\Sigma_K$ a place of $K$. We write $\Sigma_{K,f}$ and $\Sigma_{K,\infty}$ for the sets of finite and archimedean places of $K$, respectively. We write $E_v$ for the base change $E\times_K K_v$. If $v$ is a finite place of good reduction, we write $\widetilde{E}_v$ for the reduction of $E$ modulo $v$.

Given $v\in\Sigma_K$, we write $\iota_v:K\hookrightarrow\C_v$ for the associated embedding, where $\C_v$ denotes either $\C$ or $\C_p$ according as $v$ is archimedean or finite. If $y(x)=\sum_{n=0}^{\infty}a_nx^n\in K[[x]]$, we write $\iota_v(y(x)):=\sum_{n=0}^{\infty}\iota_v(a_n)x^n\in\C_v[[x]].$

	\section{Period matrices of CM elliptic curves}\label{section:background}

Let $E$ be an elliptic curve defined over a number field $K$, and let $v\in\Sigma_K$ be either an archimedean place or a finite place of good reduction of $E$.

If $v\in\Sigma_{K,f}$, we write
$$
K_{v,0}:=W(k_v)\left[\frac{1}{p}\right],
$$
where $k_v$ is the residue field of $K$ at $v$ and $p=\chara(k_v)$. We also write $q:=|k_v|=p^f$. For $v\in\Sigma_K$, we set
$$
H^1_v(E):=\begin{cases}
	H^1_{\crys}(\widetilde{E}_v/W(k_v))\otimes K_{v,0}, & \text{if } v\in\Sigma_{K,f},\\
	H^1_B(E_v^{an},\Q), & \text{if } v\in\Sigma_{K,\infty}.
\end{cases}
$$

For $v\in\Sigma_{K,f}$, Berthelot--Ogus \cite[Theorem 2.4]{bertogus} give a canonical isomorphism 
\begin{equation}\label{eq:canoniso}
 \rho_v(E):H^1_{dR}(E_v/K_v)\rightarrow H^1_v(E)\otimes_{K_{v,0}}K_v. 
\end{equation}
We refer to this as the de Rham--crystalline comparison isomorphism. For $v\in\Sigma_{K,\infty}$, we denote Grothendieck's de Rham--Betti comparison isomorphism by
\begin{equation}\label{eq:canonisobetti}
\rho_v(E):H^1_{dR}(E/K)\otimes_K\C\rightarrow H^1_v(E)\otimes_{\Q}\C.
\end{equation}

For each such $v$, the comparison isomorphism, together with choices of bases for the relevant cohomology groups, determines a period matrix. We use the following class of bases of $H^1_{dR}(E/K)$ throughout.
\begin{definition}\label{hodgebasis}
	Let $E/K$ be an elliptic curve. We call an ordered basis
	$\Gamma_{dR}(E):=(\omega,\eta)$ of $H^1_{dR}(E/K)$ a \textbf{Hodge basis} if
	\begin{enumerate}
		\item $\omega\in F^1=e^{*}\Omega_{E/K}\subset H^1_{dR}(E/K)$ is non-zero, and
		\item $\langle\omega,\eta\rangle=1$ with respect to the polarization pairing.
	\end{enumerate}
\end{definition}

For the remainder of this section, let $E/K$ be a CM elliptic curve. We assume that $E$ has everywhere good reduction, that
$$
F:=\End_{\bar{\Q}}^0(E)=\End_K^0(E)
$$
is its CM field of endomorphisms, and that $F\subset K$.\footnote{These assumptions may be achieved after replacing $K$ by a finite extension. In particular, $E$ acquires everywhere good reduction after a finite extension by \cite{serretate}. The equality $\End_K^0(E)=\End_{\bar{\Q}}^0(E)$ is equivalent to $F\subset K$; see the Remark at the beginning of $\S 10$ in \cite{langellfuns}.}

\subsection{Crystalline bases}\label{section:bases}

Throughout this subsection, $v$ denotes a finite place of $K$. We associate to each such place $v$ a symplectic basis $\Gamma_v(E)$ of $H^1_v(E)_{\C_v}$. The construction depends on whether $E$ has ordinary or supersingular reduction at $v$.

\begin{definition}
	Let $E/K$ be a CM elliptic curve as above and let $v\in\Sigma_{K,f}$. We say that $v$ is an \textbf{ordinary} (resp. \textbf{supersingular}) \textbf{place for $E$} if the reduction $\widetilde{E}_v$ of $E$ at $v$ is an ordinary (resp. supersingular) elliptic curve over $k_v$.
\end{definition}

\subsubsection{Ordinary reduction}\label{section:ordredbases}

Assume first that $v$ is an ordinary place for $E$. Let $F_v\in\End(\widetilde{E}_v)$ denote the $q$-th power Frobenius endomorphism, where $q=p^f$ is as above, and set $\phi_v$ to be the endomorphism of $H^1_v(E)$ induced by $F_v$ via pullback.

Since $\widetilde{E}_v$ is ordinary, the slopes of $\phi_v$ are $0$ and $1$. For Dieudonn\'e modules, this fact is stated in \cite[p.~98]{demazure}. The corresponding statement for $H^1_{\crys}$ follows from the comparison of Berthelot--Breen--Messing between crystalline cohomology and the Dieudonn\'e module of $\widetilde{E}_v$; see \cite{berbreenmessing}.

Let $\lambda_0$, $\lambda_1$ denote the eigenvalues of $\phi_v$, indexed so that $v(\lambda_i)=iv(q)$; see Manin's theorem in \cite[p.~98]{demazure}. Choose non-zero eigenvectors $\gamma_v$, $\delta_v\in H^1_v(E)_{\C_v}$ with eigenvalues $\lambda_1$ and $\lambda_0$, respectively. Since $\lambda_0\neq\lambda_1$, the vectors $\gamma_v$, $\delta_v$ are linearly independent, and hence
$$
\mu:=\langle\gamma_v,\delta_v\rangle\neq 0,
$$
where $\langle\cdot,\cdot\rangle$ denotes the symplectic pairing. Replacing $\delta_v$ by $\mu^{-1}\delta_v$ gives a symplectic basis of $H^1_v(E)_{\C_v}$.

\begin{definition}\label{defcanord}
	A symplectic basis
	$$
	\Gamma_v(E):=(\gamma_v,\delta_v)
	$$
	of $H^1_v(E)_{\C_v}$, obtained by the above construction up to replacing $(\gamma_v,\delta_v)$ by $(-\delta_v,\gamma_v)$, is called a \textbf{$\phi_v$-basis} of $H^1_v(E)_{\C_v}$.
\end{definition}

Although the basis $\Gamma_v(E)$ is not unique, its construction makes it compatible with homomorphisms:

\begin{lemma}\label{lemmacrysordbasis}
	Let $\widetilde{E}$, $\widetilde{E}'$ be ordinary elliptic curves over $k_v$, and $f\in\homm(\widetilde{E},\widetilde{E}')$ be a non-zero homomorphism. Let $\Gamma_v(\widetilde{E})=(\gamma,\delta)$ and $\Gamma_v(\widetilde{E}')=(\gamma',\delta')$ be $\phi_v$-bases.
	
	After replacing one of these bases by the allowed symplectic permutation if necessary, the pullback
	$$
	f_{\crys}:H^1_{\crys}(\widetilde{E}'/K_{v,0})\longrightarrow H^1_{\crys}(\widetilde{E}/K_{v,0})
	$$
	satisfies
	$$
	f_{\crys}(\gamma')=\zeta_1\cdot\gamma,\qquad f_{\crys}(\delta')=\zeta_0\cdot\delta
	$$
	for some $\zeta_0$, $\zeta_1\in\C_v^\times$.
\end{lemma}
\begin{proof} 
	After replacing $(\gamma,\delta)$, respectively $(\gamma',\delta')$, by the allowed symplectic permutation if necessary, we may assume that $\gamma$, $\gamma'$ are eigenvectors whose eigenvalues have $v$-adic valuation $v(q)$.
	
	Since $f\neq 0$, $\widetilde{E}$ and $\widetilde{E}'$ are isogenous. Hence their Frobenii have the same characteristic polynomial, and thus the same eigenvalues, which we denote by $\lambda_0$, $\lambda_1$, indexed so that $v(\lambda_i)=iv(q)$; see \cite[Proposition~8.2]{husemoller}.
	
	The pullback $f_{\crys}$ is a morphism of $F$-isocrystals. Hence
	$$
	f_{\crys}\circ\phi'_{\crys}=\phi_{\crys}\circ f_{\crys},
	$$
	where $\phi'_{\crys}$ denotes the $q$-th power Frobenius on $H^1_{\crys}(\widetilde{E}'/K_{v,0})$.
	
	Applying this identity to $\gamma'$ gives
	$$
	\phi_{\crys}\bigl(f_{\crys}(\gamma')\bigr)=f_{\crys}\bigl(\phi'_{\crys}(\gamma')\bigr)=\lambda_1 f_{\crys}(\gamma').
	$$
	Hence $f_{\crys}(\gamma')$ belongs to the $\lambda_1$-eigenspace of $\phi_{\crys}$, which is spanned by $\gamma$. Therefore,
	$$
	f_{\crys}(\gamma')=\zeta_1\gamma
	$$
	for some $\zeta_1\in\C_v^{\times}$. Similarly,
	$$
	f_{\crys}(\delta')=\zeta_0\delta
	$$
	for some $\zeta_0\in\C_v^{\times}$.
\end{proof}
	\subsubsection{Supersingular reduction}\label{section:supersingbases}
Assume that $E$ has supersingular reduction at the finite place $v$.

We use Andr\'e's construction of $p$-adic Betti lattices for abelian varieties with supersingular reduction; see \cite[\S 5]{andremots} and \cite[\S 5.3.1--5.3.5]{andremsj}. We recall the construction in the present setting.

Since $\widetilde{E}_v$ is supersingular,
$$
\End^{0}_{\overline{k_v}}(\widetilde{E}_v)=D_{p,\infty},
$$
where $D_{p,\infty}$ is the unique quaternion algebra over $\Q$ ramified precisely at $p$ and $\infty$. After a finite extension $\F_{q'}/\F_q$, all geometric endomorphisms of $\widetilde{E}_v$ are defined over $\F_{q'}$. Thus, setting
$$
\widetilde{E}':=\widetilde{E}_{v}\times_{\spec(\F_q)}\spec(\F_{q'}),
$$
we have
$$
\End^{0}_{\F_{q'}}(\widetilde{E}')=D_{p,\infty}.
$$

Let $K'_{v,0}/K_{v,0}$ be the unique unramified extension with residue field $\F_{q'}$. Then $D_{p,\infty}$ acts on
$$
H^1_{\crys}(\widetilde{E}'/K'_{v,0})=H^1_{v}(E)\otimes_{K_{v,0}}K'_{v,0}.
$$
Set $K'_v:=K'_{v,0}K_v$. After extending scalars in the Berthelot--Ogus comparison isomorphism \eqref{eq:canoniso} to $K'_v$, we obtain
\begin{equation}
	H^1_{dR}(E/K)\otimes_K K'_v	\xrightarrow{\sim}	\bigl(H^1_{v}(E)\otimes_{K_{v,0}}K'_{v,0}\bigr)	\otimes_{K'_{v,0}}K'_v.
\end{equation}

Since $F\hookrightarrow D_{p,\infty}$, we have
$$
D_{p,\infty}\otimes_{\Q}F\simeq M_2(F);
$$
see \cite[Proposition~1.2.3]{szamuely}. Fix the embedding $F\hookrightarrow\C_v$ induced by the chosen embedding $K\hookrightarrow\C_v$. Following \cite[\S~5.3.2--5.3.3]{andremsj}, choose a non-zero vector $u$ such that
$$
H^1_B(\widetilde{E}_v,F):=M_2(F)\cdot u
$$
is the corresponding $p$-adic Betti lattice. In particular, $H^1_B(\widetilde{E}_v,F)$ is a two-dimensional $F$-vector space contained in
$$
H^1_v(E)\otimes_{K_{v,0}}\C_v=H^1_{\crys}(\widetilde{E}'/K'_{v,0})\otimes_{K'_{v,0}}\C_v
$$
and is stable under the action of $D_{p,\infty}$.

The subfield $F\subset D_{p,\infty}$, viewed as a subalgebra of $\End^{0}_{\F_{q'}}(\widetilde{E}')$, acts on $H^1_B(\widetilde{E}_v,F)$. Let $\sigma$, $\tau:F\hookrightarrow\C_v$ be the two embeddings, with $\sigma$ the embedding fixed above. We obtain an eigenspace decomposition
$$
H^1_B(\widetilde{E}_v,F)=H^1_B(\widetilde{E}_v,F)^{\sigma}\oplus H^1_B(\widetilde{E}_v,F)^{\tau},
$$
where $F$ acts on $H^1_B(\widetilde{E}_v,F)^{\chi}$ through $\chi$. By the Skolem--Noether theorem, both eigenspaces are one-dimensional. Choose non-zero vectors
$$
\gamma_v\in H^1_B(\widetilde{E}_v,F)^{\sigma},\qquad\delta_v\in H^1_B(\widetilde{E}_v,F)^{\tau},
$$
and rescale $\delta_v$ so that
$$
\langle\gamma_v,\delta_v\rangle=1,
$$
where $\langle\cdot,\cdot\rangle$ denotes the polarization pairing.

\begin{definition}\label{defcansuper}
	We call an ordered basis
	$$
	\Gamma_v(E):=(\gamma_v,\delta_v)
	$$
	of $H^1_v(E)_{\C_v}$, obtained by the above construction up to replacing $(\gamma_v,\delta_v)$ by $(-\delta_v,\gamma_v)$, a \textbf{$v$-adic Betti basis} of $H^1_v(E)_{\C_v}$.
\end{definition}
	\subsection{Period relations}\label{section:cmperiods}

The following lemma gives the period relations for CM elliptic curves in the archimedean and $p$-adic settings.

\begin{lemma}\label{lemmacmperiods}
	Let $E$, $F$ be as above. There exists a Hodge basis $\Gamma_{dR}(E)$ of $H^1_{dR}(E/K)$ such that, for every place $v\in\Sigma_K$, the following hold:
	\begin{enumerate}
	\item If $v\in\Sigma_{K,\infty}$, there exists $\varpi_v\in\C_v^\times$ such that, with respect to $\Gamma_{dR}(E)$ and a symplectic basis of $H^1_v(E)_{\C}$ consisting of $F$-eigenvectors, the period matrix of $E$ is
	\begin{equation}\label{eq:cmperarch}
		\begin{pmatrix}\frac{\varpi_v}{2\pi i} & 0\\0 & \varpi_v^{-1}\end{pmatrix}.
	\end{equation}
	
	\item If $v\in\Sigma_{K,f}$, there exists $\varpi_v\in\C_v^\times$ such that, with respect to $\Gamma_{dR}(E)$ and the basis $\Gamma_v(E)$ constructed in \Cref{defcanord,defcansuper}, according to the reduction type of $E$ at $v$, the period matrix of $E$ is
		\begin{equation}\label{eq:cmpernonarch}
			\begin{pmatrix}	\varpi_v & 0\\	0 & \varpi_v^{-1}\end{pmatrix}.
		\end{equation}
	\end{enumerate}
\end{lemma}
\begin{proof}
	The archimedean case is proved in \cite[Lemma~2.9]{papaseffbrsieg}. Assume henceforth that $v\in\Sigma_{K,f}$.
	
	Let $V_{dR}:=H^1_{dR}(E/K)$ and write $F=\Q(\sqrt{-d})$, with $d\in\Z_{>0}$ square-free. Set $e:=\sqrt{-d}\in F=\End^0(E)$, and denote by
	$$
	e_{dR}\in\End(V_{dR}),\qquad
	e_{\crys}\in\End\bigl(H^1_v(E)\bigr)
	$$
	the endomorphisms induced by $e$ on de Rham and crystalline cohomology, respectively.
	
	The action of $e$ gives decompositions
	\begin{equation}\label{eq:splitdr}
		V_{dR}=V^{+}_{dR}\oplus V^{-}_{dR}
	\end{equation}
	and
	\begin{equation}\label{eq:splitstar}
		H^1_v(E)\otimes_{K_{v,0}}\C_v=W^{+}_v\oplus W^{-}_v.
	\end{equation}
	Here $V^{\pm}_{dR}$ are the eigenspaces of $e_{dR}$ with eigenvalues $\pm\sqrt{-d}$, while $W^{\pm}_v$ are the corresponding eigenspaces of $e_{\crys}$ after extending scalars to $\C_v$.
	
	As in \cite[Lemma~2.9]{papaseffbrsieg}, choose a Hodge basis
	$$
	\Gamma_{dR}(E):=(\omega,\eta)
	$$
	of $V_{dR}$ consisting of eigenvectors for the action of $F$. Choosing the sign of $\sqrt{-d}$ appropriately, we may assume that
	$$
	\omega\in V^+_{dR},\qquad \eta\in V^-_{dR}.
	$$
	
	Let $\Gamma_v(E)=(\gamma_v,\delta_v)$ be the basis constructed in \Cref{defcanord,defcansuper}, according to the reduction type of $E$ at $v$. The vectors $\gamma_v$, $\delta_v$ are eigenvectors for the action of $F$. In the supersingular case this follows directly from the construction in \Cref{section:supersingbases}. In the ordinary case, $e_{\crys}$ commutes with $\phi_v$ and therefore preserves the two Frobenius eigenspaces defining $\Gamma_v(E)$. Hence $\gamma_v$, $\delta_v$ are eigenvectors for $e_{\crys}$.
	
	After replacing $(\gamma_v,\delta_v)$ by $(-\delta_v,\gamma_v)$ if necessary, we may assume that
	$$
	\gamma_v\in W^+_v,\qquad \delta_v\in W^-_v,
	$$
	while preserving the symplecticity of the basis.

	Let $A_v=(a_{i,j})$ be the period matrix of $\rho_v(E)$ with respect to $\Gamma_{dR}(E)$ and $\Gamma_v(E)$. By compatibility of the $F$-action with $\rho_v(E)$,
	$$
	\begin{aligned}
		\rho_v(E)\bigl(e_{dR}(\omega)\bigr)	&=\sqrt{-d}\,\rho_v(E)(\omega)\\
		&=a_{1,1}\sqrt{-d}\,\gamma_v+a_{1,2}\sqrt{-d}\,\delta_v,
	\end{aligned}
	$$
	while
	$$
	e_{\crys}\bigl(\rho_v(E)(\omega)\bigr)=a_{1,1}\sqrt{-d}\,\gamma_v-a_{1,2}\sqrt{-d}\,\delta_v.
	$$
	Hence $a_{1,2}=0$. Applying the same argument to $\eta$ gives $a_{2,1}=0$.
	
	The compatibility of the polarization pairing with the de Rham--crystalline comparison isomorphism gives
	$$
	\bigl\langle \rho_v(E)(\omega),\rho_v(E)(\eta)\bigr\rangle=\langle\omega,\eta\rangle=1;
	$$
	see \cite[Chapter~5]{berbreenmessing}. Since $A_v$ is diagonal and both chosen bases are symplectic, it follows that
	$$
	a_{1,1}a_{2,2}=1.
	$$
	Thus, setting $\varpi_v:=a_{1,1}$, we obtain
	$$
	A_v=\begin{pmatrix}\varpi_v&0\\0&\varpi_v^{-1}\end{pmatrix}.
	$$
	\end{proof}

   \section{Background on the G-functions method}\label{section:backgroundgfuns}

This section recalls the G-functions framework associated to a one-parameter family of elliptic curves, in the form needed below. We use the terminology and constructions of \cite{papaspadicpart1}.

Throughout this section, fix a one-parameter family of elliptic curves $f:\CE\rightarrow S$ defined over a number field $K$, and let $s_0\in S(K)$ be such that $E_0:=\CE_{s_0}$ has complex multiplication. Assume that the induced map $S\rightarrow Y(1)$ is non-constant.

The G-functions method is used below to obtain height bounds. For this purpose, replacing the base field $K$ by a finite extension and the base curve $S$ by a finite cover does not alter the form of the bounds sought.

After replacing $K$ by a finite extension, still denoted by $K$, one may choose a finite cover $c:S'\rightarrow S$ such that the G-functions associated to the pullback family $f':\CE':=\CE\times_S S'\rightarrow S'$ have the structural properties discussed below. Write
$$
c^{-1}(s_0)=\{s_1,\ldots,s_n\}.
$$

As in \Cref{section:background}, after a further finite extension of $K$, we may assume that $E_0$ has everywhere good reduction and that
$$
F_0:=\End^0_{\bar{\Q}}(E_0)=\End^0_K(E_0).
$$
Since $\CE'_{s_i}\simeq E_0$ for every $1\leq i\leq n$, the fibers $\CE'_{s_i}$ satisfy the same assumptions. The cover $c:S'\rightarrow S$ is chosen so that there exists $x\in K(S')$ with simple zeros precisely at $s_1,\ldots,s_n$ and with the additional properties established in \cite[Lemma~5.1]{daworr4}.

\subsection{G-functions and places of proximity}\label{section:gfunctionsproximity}

Since $x$ has a simple zero at each $s_i$, it is a local parameter there. Choose an affine open neighborhood $U\subset S$ of $s_0$ on which $H^1_{dR}(\CE/S)$ admits a basis of sections, and set $U':=c^{-1}(U)$. Its pullback gives a basis of $H^1_{dR}(\CE'/S')$ on $U'$. With respect to this basis, the Gauss--Manin connection determines a system of differential equations whose normalized fundamental matrix solutions at the points $s_i$ have entries in $K[[x]]$.

By \cite[Chapter~3]{andre1989g}, the entries of these matrices are G-functions. Following \cite{daworr4}, these matrices are grouped according to an equivalence relation determined by the existence of global isogenies. Let $\Lambda\subset \{s_1,\ldots,s_n\}$ be a set of representatives for these equivalence classes. Choosing one matrix for each $\lambda\in\Lambda$, we obtain a family
$$
\mathcal{Y}:=\{\mathcal{Y}_{\lambda}(x)\}_{\lambda\in\Lambda},\qquad \mathcal{Y}_{\lambda}(x)\in\GL_2(K[[x]]).
$$

For each $\lambda\in\Lambda$ and $v\in\Sigma_K$, the entries of $\mathcal{Y}_\lambda(x)$ converge on a sufficiently small $v$-adic analytic disc $\Delta_{v,\lambda}$ in the analytification of $S'$ centered at $\lambda$. Thus, if $s\in S'(\bar{\Q})$ and $w$ is a place of $K(s)$ above $v$, the value
$$
\iota_w\big(\mathcal{Y}_\lambda(x(s))\big)
$$
is defined whenever $s\in\Delta_{v,\lambda}$.

For finite places $v$, the discs $\Delta_{v,\lambda}$ may be chosen so that, if $s\in S'(\bar{\Q})\cap\Delta_{v,\lambda}$ and $w\in\Sigma_{K(s)}$ lies above $v$, then the fibers $\CE'_s$ and $\CE'_{\lambda}$ have the same reduction at $w$. We refer to \cite{papaspadicpart1} for this refinement.

\begin{definition}\label{defnvadicprox}
	Let $s\in S'(\bar{\Q})$, let $w\in\Sigma_{K(s)}$, and let $v\in\Sigma_K$ be the place below $w$. We say that $s$ is $w$-adically close to $\lambda\in\Lambda$ if $s\in\Delta_{v,\lambda}$.
\end{definition}

\subsection{Period matrices and functional relations}\label{section:gfunctionsperiods}

The relative comparison isomorphisms give a cohomological interpretation of the values of the matrices $\mathcal{Y}_\lambda$ at points in $\Delta_{v,\lambda}$. At archimedean places we use the relative de Rham--Betti comparison isomorphism, and at finite places the relative de Rham--crystalline comparison isomorphism; see \cite{andre1989g,andremots,bertogus,ogus}. These are the relative counterparts of the comparison isomorphisms considered in \Cref{section:background}. The ramified finite-place case is handled in \cite{daworrpap,papaspadicpart1}.

Let $s\in\Delta_{v,\lambda}\cap S'(\bar{\Q})$, let $w\in\Sigma_{K(s)}$ be a place above $v$, and let $\CP_w(s)$ and $\CP_w(\lambda)$ denote the corresponding period matrices. Then
$$
\CP_w(s)=\iota_w\big(\mathcal{Y}_\lambda(x(s))\big)\CP_w(\lambda).
$$
These period matrices are defined using the chosen basis of sections of $H^1_{dR}(\CE'/S')$ together with a basis of $H^1_v(\CE'_\lambda)$. The latter is identified with the space of horizontal sections of $(H^1_{dR}(\CE'/S'),\nabla)$ on $\Delta_{v,\lambda}$. We refer to \cite{papaspadicpart1} for the full construction.

The relations among the values of $\mathcal{Y}$ constructed in \Cref{section:cmpadicrels} must be distinguished from the functional relations satisfied identically by $\mathcal{Y}$. In the terminology of \cite[Chapter~VII.5]{andre1989g}, the former are required to be non-trivial.

For the applications below, it suffices to consider the functional relations in the case where $\mathcal{Y}$ consists of a single $2\times 2$ matrix of power series. We record this case here; the general setting is treated in \cite{daworrpap,papaspadicpart1}.
\begin{prop}\label{trivialrelations}
	Assume that the family $\mathcal{Y}$ of G-functions associated to the setup above consists of a single $2\times 2$ matrix of power series $\mathcal{Y}(x)$. Then $\mathcal{Y}(x)^{\zar}\subset (M_{2\times 2})_{\bar{\Q}(x)}$ is the subvariety cut out by the ideal
	\begin{equation}\label{eq:trivialideal}
		I(\SL_2):=\left\langle X_{1,1}X_{2,2}-X_{1,2}X_{2,1}-1\right\rangle,
	\end{equation}
	where $M_{2\times 2}=\spec\bigl(\Q[X_{i,j}:1\leq i,j\leq 2]\bigr)$.
\end{prop}
This statement is mentioned without proof in \cite{beukers}. For a proof, see \S 4.5.6 of \cite{daworrpap}.
   
   \section{Relations among values of G-functions}\label{section:cmpadicrels}

This section contains the main technical construction of the paper and proves \Cref{naiverelations}. We construct relations among the values of the G-functions associated to a one-parameter family $f:\CE\rightarrow S$ of elliptic curves at points $s\in S(\bar{\Q})$ for which the fiber $\CE_s$ has complex multiplication. The construction depends on the place $v$ and, at finite places, on the reduction type of the central fiber $E_0$.

\subsection{Notation}\label{section:finrelsnotation}

Fix $\lambda\in\Lambda$ as in \Cref{section:gfunctionsproximity}. The constructions below are local at $\lambda$. After replacing $K$ by a finite extension and shrinking $S'$ around $\lambda$ if necessary, we suppress the pullback notation and write $S,\CE,s_0,E_0$ and $Y_G$ for $S',\CE',\lambda,\CE'_{\lambda}$ and $\mathcal{Y}_{\lambda}$, respectively. Thus $E_0$ has everywhere good reduction,
$$
\End^0_{\bar{\Q}}(E_0)=\End_K^0(E_0)=F_0,
$$
and $x$ has a unique simple zero at $s_0$. Throughout this section, we use the notion of $v$-adic proximity from \Cref{defnvadicprox}.

The action of $F_0$ on $H^1_{dR}(E_0/K)$ gives a decomposition into two one-dimensional eigenspaces. Since $F^1H^1_{dR}(E_0/K)$ is the space of invariant differentials and is stable under $F_0$, it coincides with one of these eigenspaces. After shrinking $S$ around $s_0$ if necessary, we may therefore choose a symplectic basis $(\omega,\eta)$ of $H^1_{dR}(\CE/S)$ such that $\omega$ is a section of $F^1H^1_{dR}(\CE/S)$ and $\omega_0:=\omega(s_0)$, $\eta_0:=\eta(s_0)$ are $F_0$-eigenvectors, as in the proof of \Cref{lemmacmperiods}. By the construction of \Cref{section:gfunctionsproximity}, this basis determines a matrix $Y_G\in\SL_2(\bar{\Q}[[x]])$ of G-functions.

\subsection{Relations at ordinary places}\label{section:cmordpadic}

We first consider places that are ordinary for the central fiber $E_0$.
\begin{prop}\label{thmordrels}
	With notation as in \Cref{section:finrelsnotation}, let $s\in S(\bar{\Q})$ be such that $\CE_s$ is CM. 
	
	Then, for every finite place $v\in\Sigma_{K(s)}$ that is ordinary for $E_0$ and such that $s$ is $v$-adically close to $s_0$, we have
	$$
	\iota_v(Y_{1,2}(x(s)))=0.
	$$
\end{prop}
\begin{proof}
	Since $s$ is $v$-adically close to $s_0$, the fibers $\CE_s$ and $E_0$ have the same reduction at $v$. Specialization therefore embeds both $F_s:=\End^0_{\bar{\Q}}(\CE_s)$ and $F_0$ into $\End^0_{\bar{\mathbb{F}}_q}(\widetilde{E}_{0,v})$. Since the common reduction is ordinary, this endomorphism algebra is a quadratic field. Both specialization maps are therefore isomorphisms, and hence $F_s=F_0$.
	
	Let $\Gamma_v(E_0)=(\gamma,\delta)$ be a $\phi_v$-basis as in \Cref{defcanord}. By \Cref{lemmacmperiods}, there exists $\varpi_v\in\C_v^\times$ such that
	$$
	\CP_v(s_0)=	\begin{pmatrix}	\varpi_v&0\\0&\varpi_v^{-1}	\end{pmatrix}.
	$$
	Hence, by \Cref{section:gfunctionsperiods},
	$$
	\CP_v(s)=\iota_v(Y_G(x(s)))\begin{pmatrix}\varpi_v&0\\0&\varpi_v^{-1}\end{pmatrix}.
	$$
	
	Since $F_s=F_0$, the CM elliptic curves $\CE_s$ and $E_0$ are isogenous. Fix an isogeny $\phi:\CE_s\rightarrow E_0$. After replacing $K(s)$ by a finite extension, assume that $\phi$ is defined over $K(s)$. Let $\phi_{dR}$ denote the induced pullback on de Rham cohomology. Let $\widetilde{\phi}_v$ denote the reduction of $\phi$ at $v$, viewed as an endomorphism of the common reduction $\widetilde{E}_{0,v}$. Write
	$$
	\phi_{\crys}:H^1_{\crys}(\widetilde{E}_{0,v}/K(s)_{v,0})\longrightarrow	H^1_{\crys}(\widetilde{E}_{0,v}/K(s)_{v,0})
	$$
	for the induced pullback.
	
	Functoriality of the de Rham--crystalline comparison isomorphism gives
	\begin{equation}\label{eq:isogcompat}
		\CP_v(s)\circ\phi_{dR}=\phi_{\crys}\circ\CP_v(s_0).
	\end{equation}
	Since $\phi_{\crys}$ is induced by the endomorphism $\widetilde{\phi}_v$ of the common reduction, \Cref{lemmacrysordbasis} gives
	$$
	[\phi_{\crys}]=\begin{pmatrix}\zeta_v&0\\0&\xi_v\end{pmatrix}
	$$
	with respect to the fixed $\phi_v$-basis, for some $\zeta_v,\xi_v\in\C_v^\times$.
	
	Since $\phi_{dR}$ preserves the Hodge filtration, $\phi_{dR}\omega_0=a\omega_s$ and $\phi_{dR}\eta_0=c\omega_s+d\eta_s$, with $a,d\neq 0$. Thus \eqref{eq:isogcompat} becomes
	\begin{equation}\label{eq:isogord1}	
		\begin{pmatrix}\iota_v(a)&0\\ \iota_v(c)&\iota_v(d)\end{pmatrix}\CP_v(s)=\CP_v(s_0)\begin{pmatrix}\zeta_v&0\\0&\xi_v\end{pmatrix}
	\end{equation}
	
	Set $y_{i,j}:=\iota_v(Y_{i,j}(x(s)))$. Since $\CP_v(s)=(y_{i,j})\CP_v(s_0)$, \eqref{eq:isogord1} gives
	\begin{equation}
		\begin{pmatrix}\iota_v(a)&0\\ \iota_v(c)&\iota_v(d)\end{pmatrix}\cdot (y_{i,j})=\CP_v(s_0)\cdot \begin{pmatrix}\zeta_v&0\\0&\xi_v\end{pmatrix}\cdot \CP_v(s_0)^{-1}.
	\end{equation}
	Since $\CP_v(s_0)=\diag(\varpi_v,\varpi_v^{-1})$ is diagonal, the right-hand side simplifies to
	\begin{equation}\label{eq: ordinary relation recovering beukers}
		\begin{pmatrix}\iota_v(a)&0\\\iota_v({c})&\iota_v(d)\end{pmatrix}\cdot (y_{i,j})=\begin{pmatrix}\zeta_v&0\\0&\xi_v\end{pmatrix}.
	\end{equation}
	Since $a\neq 0$, comparing the $(1,2)$-entries gives $y_{1,2}=0$.
\end{proof}

\begin{remark}\label{remark ordinary recover beukers}
	Note that the characteristic polynomial of the matrix $[\phi_{\crys}]$ is the same as that of the endomorphism $\widetilde{\phi}_v$ of $\widetilde{E}_{0,v}$; see \cite[Chapter~V]{demazure}. Since the latter has coefficients in $\Z$, so does the former. Together with \eqref{eq: ordinary relation recovering beukers}, this recovers \cite[Theorem~5.4]{beukers}.
\end{remark}
   \subsection{Relations at supersingular places}\label{section:cmsuperpadic}

We now consider places that are supersingular for the central fiber $E_0$. For a CM point $s$, write $F_s:=\End_{\bar{\Q}}^0(\CE_s)$. In contrast with the ordinary case, $v$-adic proximity of $s$ to $s_0$ does not force $F_s=F_0$. We therefore treat separately the cases $F_s\neq F_0$ and $F_s=F_0$.

Set
$$
\Sigma_{\ssing}(s_0):=\{w\in\Sigma_{K,f}:w\text{ is supersingular for }E_0\}.
$$
For such a CM point $s\in S(\bar{\Q})$, set
$$
\begin{aligned}
	\CP_{\ssing}(s,s_0):=\{v\in\Sigma_{K(s),f}:{}&v|w\text{ for some }w\in\Sigma_{\ssing}(s_0),\\
	&\text{and }s\text{ is }v\text{-adically close to }s_0\}.
\end{aligned}
$$

\subsubsection{The case $F_s\neq F_0$}\label{section: andre construction}

The following proposition is based on a local construction of Andr\'e \cite{andremots,andremsj}.

\begin{prop}\label{propandresupercm}
	Let $s\in S(\bar{\Q})$ be such that $\CE_s$ has complex multiplication and $F_s\neq F_0$.
	
	Then there exists a polynomial $R_{s,\ssing}\in \bar{\Q}[X_{i,j}:1\leq i,j\leq 2]$ such that:
	\begin{enumerate}
		\item $R_{s,\ssing}$ has coefficients in $L_s:=K(s)F_s$,
		
		\item $\iota_v(R_{s,\ssing}(Y_G(x(s))))=0$ for every $v\in\Sigma_{L_s,f}$ lying above a place in $\CP_{\ssing}(s,s_0)$,
		
		\item $R_{s,\ssing}$ is homogeneous and
		$$
		\deg(R_{s,\ssing})\leq 4|\CP_{\ssing}(s,s_0)|,
		$$
		and
		
		\item $R_{s,\ssing}\notin I(\SL_2)$.
	\end{enumerate}
\end{prop}
\begin{proof}
	Write $F_s=\Q(\sqrt{-d_s})$ and $F_0=\Q(\sqrt{-d_0})$, with $d_s,d_0>0$ square-free. By \cite[\S~10]{langellfuns}, all endomorphisms of $\CE_s$ are defined over $L_s=K(s)F_s$.
	
	Fix a place $v\in\Sigma_{L_s,f}$ lying above a place in $\CP_{\ssing}(s,s_0)$. The common supersingular reduction of $\CE_s$ and $E_0$ has geometric endomorphism algebra $D_{p,\infty}$, and specialization gives embeddings
	$$
	F_s,F_0\hookrightarrow D_{p,\infty}.
	$$
	Since $F_s\neq F_0$, these determine distinct quadratic subfields of $D_{p,\infty}$.
	
	By Andr\'e's calculation in \cite[\S 5.c]{andremots}, see also \cite[\S I.5.3.7]{andremsj},\footnote{The hypothesis in \cite[\S 5.c]{andremots} that the residue characteristic be unramified in $F_0$ is not used in the derivation of this relation. In \cite[\S I.5.3.7]{andremsj}, the same calculation is carried out for two distinct quadratic subfields of $D_{p,\infty}$; unramifiedness enters only in the subsequent choice of the sign of $\sqrt{d_sd_0}$ satisfying a congruence modulo $p$.} there exist $a_s\in L_s$ and $m_{s,v}\in\Z$ such that the quadratic polynomial
	\begin{equation}\label{eq:andresupersingularpoly}
		\begin{aligned}
			R_{s,v}:={}&(m_{s,v}+2\sqrt{d_sd_0})X_{1,1}(X_{2,2}+a_sX_{1,2})\\
			&-(m_{s,v}-2\sqrt{d_sd_0})X_{1,2}(X_{2,1}+a_sX_{1,1}).
		\end{aligned} 
	\end{equation}
	satisfies $\iota_v(R_{s,v}(Y_G(x(s))))=0$.
	
	Consider the matrices
	$$
	S(n,l):=\begin{pmatrix}n&l\\0&\frac{1}{n}\end{pmatrix}\in \SL_2(\C)
	\qquad\text{and}\qquad
	T(n):=\begin{pmatrix}0&\frac{1}{n}\\-n&0\end{pmatrix}\in \SL_2(\C).
	$$
	Suppose that $R_{s,v}\in I(\SL_2)$. Then $R_{s,v}$ vanishes identically on $\SL_2(\C)$. Evaluating at $S(n,l)$ gives
	$$
	nl(4a_s\sqrt{d_0d_s})+(m_{s,v}+2\sqrt{d_0d_s})=0,
	$$
	while evaluating at $T(n)$ gives
	$$
	m_{s,v}-2\sqrt{d_0d_s}=0.
	$$
	Since the first identity holds for all $n,l\in\C$, we obtain
	$$
	a_s=0
	\qquad\text{and}\qquad
	m_{s,v}+2\sqrt{d_0d_s}=0.
	$$
	Together with the second identity, this gives a contradiction since $d_0d_s\neq0$.
	
	Define
	$$
	R_{s,\ssing}:=\prod_v R_{s,v},
	$$
	where the product ranges over all $v\in\Sigma_{L_s,f}$ lying above a place in $\CP_{\ssing}(s,s_0)$. By construction, $\iota_v(R_{s,\ssing}(Y_G(x(s))))=0$ for every such $v$.
	
	Each $R_{s,v}$ is homogeneous of degree $2$. Since $[L_s:K(s)]\leq 2$, the number of places occurring in the product is at most $2|\CP_{\ssing}(s,s_0)|$. Hence
	$$
	\deg(R_{s,\ssing})\leq 4|\CP_{\ssing}(s,s_0)|.
	$$
	Finally, $I(\SL_2)$ is prime, and none of the factors $R_{s,v}$ belongs to $I(\SL_2)$. Hence $R_{s,\ssing}\notin I(\SL_2)$.
\end{proof}
   \subsubsection{The case $F_s=F_0$}\label{section:samecmfield}

\begin{prop}\label{propsamecmfield}
	Let $s\in S(\bar{\Q})$, $s\neq s_0$, be such that $\CE_s$ has complex multiplication and $F_s=F_0$.
	
	Then there exist a finite extension $L_s/K(s)$ with $[L_s:K(s)]\leq 12$ and a polynomial $R_{s,\ssing}\in \bar{\Q}[X_{i,j}:1\leq i,j\leq 2]$ such that:
	\begin{enumerate}
		\item $R_{s,\ssing}$ has coefficients in $L_s$,
		
		\item $\iota_v(R_{s,\ssing}(Y_G(x(s))))=0$ for every $v\in\Sigma_{L_s,f}$ lying above a place in $\CP_{\ssing}(s,s_0)$,
		
		\item $R_{s,\ssing}$ is homogeneous and
		$$
		\deg(R_{s,\ssing})\leq 24|\CP_{\ssing}(s,s_0)|,
		$$
		and
		
		\item $R_{s,\ssing}\notin I(\SL_2)$.
	\end{enumerate}
\end{prop}
\begin{proof}
	Since $F_s=F_0$, the elliptic curves $\CE_s$ and $E_0$ are isogenous over $\bar{\Q}$. By \cite[p.~18]{masserwustholzisogenies}, there exists an extension $L_s/K(s)$ with $[L_s:K(s)]\leq 12$ over which an isogeny $\phi:\CE_s\rightarrow E_0$ is defined. Fix a place $v\in\Sigma_{L_s,f}$ lying above a place in $\CP_{\ssing}(s,s_0)$.
	
	Let $\phi_{dR}$ and $\phi_{\crys}$ denote the pullbacks induced by $\phi$ on de Rham and crystalline cohomology, respectively. The latter is induced by the reduction $\widetilde{\phi}_v$ of $\phi$ modulo $v$. Since $s$ is $v$-adically close to $s_0$, $\widetilde{\CE}_{s,v}=\widetilde{E}_{0,v}$, and hence
	$$
	\widetilde{\phi}_v\in\End(\widetilde{E}_{0,v}).
	$$
	
	By functoriality of the de Rham--crystalline comparison isomorphism, \eqref{eq:isogcompat} gives
	\begin{equation}\label{eq:isogram1}
		[\phi]_{dR}\CP_v(s)=\CP_v(s_0)[\phi]_{\crys}.
	\end{equation}
	The matrices $[\phi]_{dR}$ and $[\phi]_{\crys}$ are taken with respect to the chosen de Rham and crystalline bases.
	
	Since $\phi_{dR}$ preserves the Hodge filtration, we have
	$$
	[\phi]_{dR}=\begin{pmatrix}\iota_v(a)&0\\ \iota_v(c)&\iota_v(d)\end{pmatrix}
	$$
	for some $a,c,d\in L_s$, with $a,d\neq 0$. Set $y_{i,j}:=\iota_v(Y_{i,j}(x(s)))$. Since $\CP_v(s)=(y_{i,j})\CP_v(s_0)$, \eqref{eq:isogram1} becomes
	\begin{equation}\label{eq:ramif2}
		\begin{pmatrix}\iota_v(a)&0\\ \iota_v(c)&\iota_v(d)\end{pmatrix}(y_{i,j})=\CP_v(s_0)[\phi]_{\crys}\CP_v(s_0)^{-1}.
	\end{equation}
	
	Taking traces in \eqref{eq:ramif2} gives
	\begin{equation}\label{eq:ramif3}
		\iota_v(a Y_{1,1}(x(s))+c Y_{1,2}(x(s))+dY_{2,2}(x(s)))=t_v:=\tr([\phi]_{\crys}).
	\end{equation}
	
	Since $\phi_{\crys}$ is induced by the endomorphism $\widetilde{\phi}_v$ of $\widetilde{E}_{0,v}$, its trace is the reduced trace of $\widetilde{\phi}_v$ and thus lies in $\Q$; see \cite[Chapter~V]{demazure}.
	
	Define
	$$
	R_{s,v}:=(aX_{1,1}+cX_{1,2}+dX_{2,2})^2-t_v^2(X_{1,1}X_{2,2}-X_{1,2}X_{2,1}).
	$$
	Then $R_{s,v}$ is homogeneous of degree $2$ and has coefficients in $L_s$. By \eqref{eq:ramif3} and the identity $\det(Y_G)=1$, we have $\iota_v(R_{s,v}(Y_G(x(s))))=0$.
	
	Define
	$$
	R_{s,\ssing}:=\prod_v R_{s,v},
	$$
	where the product ranges over all $v\in\Sigma_{L_s,f}$ lying above a place in $\CP_{\ssing}(s,s_0)$. Since $[L_s:K(s)]\leq 12$ and each $R_{s,v}$ has degree $2$, we have
	$$
	\deg(R_{s,\ssing})\leq 24|\CP_{\ssing}(s,s_0)|.
	$$
	By construction, $\iota_v(R_{s,\ssing}(Y_G(x(s))))=0$ for every $v\in\Sigma_{L_s,f}$ lying above a place in $\CP_{\ssing}(s,s_0)$.
	
	Since $I(\SL_2)$ is prime, it suffices to show that $R_{s,v}\notin I(\SL_2)$ for every $v$ occurring in the product. Suppose that $R_{s,v}\in I(\SL_2)$ for some $v$. Evaluating $R_{s,v}$ at the matrices $S(n,l)$ from the proof of \Cref{propandresupercm} and varying $n,l\in\C$ gives $a=0$, contradicting the invertibility of $[\phi]_{dR}$.
\end{proof}
\begin{remark}
	The construction does not use any special property of the chosen basis of $H^1_v(E_0)$. The relation is obtained from the isogeny $\phi:\CE_s\rightarrow E_0$ and the functoriality of the de Rham--crystalline comparison isomorphism.
\end{remark}
   \subsection{Relations at archimedean places}\label{section:cmarchimedean}

The following proposition is based on a local construction of Beukers; see \cite[Theorem~3.3]{beukers}.

\begin{prop}\label{propbeukersarch}
	Let $s\in S(\bar{\Q})$ be such that $\CE_s$ has complex multiplication.
	
	Then there exists a polynomial $R_{s,\arch}\in \bar{\Q}[X_{i,j}:1\leq i,j\leq 2]$ such that:
	\begin{enumerate}
		\item $R_{s,\arch}$ has coefficients in $L_s:=K(s)F_s$,
		
		\item $\iota_v(R_{s,\arch}(Y_G(x(s))))=0$ for every $v\in\Sigma_{L_s,\infty}$ such that $s$ is $v$-adically close to $s_0$,
		
		\item $R_{s,\arch}$ is homogeneous and
		$$
		\deg(R_{s,\arch})\leq 4[K(s):\Q],
		$$
		and
		
		\item $R_{s,\arch}\notin I(\SL_2)$.
	\end{enumerate}
\end{prop}
\begin{proof}
	Fix $v\in\Sigma_{L_s,\infty}$ such that $s$ is $v$-adically close to $s_0$. By \cite[Theorem~3.3]{beukers}, there exist $a_s,m_{s,v}\in L_s$ such that 
	$$
	R_{s,v}:=X_{1,1}(X_{2,2}+a_sX_{1,2})-m_{s,v}(X_{1,1}X_{2,2}-X_{1,2}X_{2,1})
	$$
	satisfies $\iota_v(R_{s,v}(Y_G(x(s))))=0$.
	
	Suppose that $R_{s,v}\in I(\SL_2)$. Evaluating $R_{s,v}$ at the matrices $S(n,l)$ from the proof of \Cref{propandresupercm} gives
	$$
	1+a_snl-m_{s,v}=0
	$$
	for all $n,l\in\C$. Hence $a_s=0$ and $m_{s,v}=1$. Thus $X_{1,2}X_{2,1}$ belongs to the ideal, a contradiction.
	
	Define
	$$
	R_{s,\arch}:=\prod_v R_{s,v},
	$$
	where the product ranges over all $v\in\Sigma_{L_s,\infty}$ such that $s$ is $v$-adically close to $s_0$. Since $[L_s:K(s)]\leq 2$, we have $\deg(R_{s,\arch})\leq 4[K(s):\Q]$.	Finally, $I(\SL_2)$ is prime and none of the factors $R_{s,v}$ belongs to $I(\SL_2)$. Hence $R_{s,\arch}\notin I(\SL_2)$.
\end{proof}
   \section{Height bounds and applications}\label{section:applications}

We apply Andr\'e--Bombieri's ``Hasse principle for the values of G-functions'' to obtain the height bounds discussed in \Cref{section:heightboundsintro}. We then isolate the remaining obstruction to deriving an effective version of Siegel's theorem \cite{siegelog} from these bounds.

\subsection{Height bounds}\label{section: cmheightboundproof}

We treat separately the cases in which the CM fields of the singular moduli coincide with, or are distinct from, that of a fixed CM elliptic curve.

\subsubsection{Singular moduli with the same CM field}\label{section: same cm field heights}

We begin with singular moduli whose CM field coincides with that of a fixed CM elliptic curve.

\begin{prop}\label{propisogenouscm}
	Let $E_0/K$ be a CM elliptic curve with $j$-invariant $j_0$, and set
	$$
	F_0:=\End^{0}_{\bar{\Q}}(E_0).
	$$
	Define
	$$
	A(j_0):=\{j\in\bar{\Q}:\End^{0}_{\bar{\Q}}(E_j)=F_0\},
	$$
	where $E_j$ denotes an elliptic curve with $j$-invariant $j$.
	
	There exist effectively computable constants $C_1$, $C_2>0$, depending only on $K$ and $j_0$, such that
	$$
	h(j)\leq C_1+C_2\log[\Q(j):\Q]
	$$
	for every $j\in A(j_0)$.
\end{prop}
\begin{proof} 
	After replacing $K$ by a finite extension of degree depending only on $j_0$, we may assume that $E_0$ has everywhere good reduction; see \cite{serretate}.
	
	Let $K(j)$ denote the compositum of $K$ and $\Q(j)$. By \cite{silverberg}, there exists an extension $K'(j)/K(j)$ of absolutely bounded degree such that
	$$
	\End^{0}_{K'(j)}(E_j)=\End^{0}_{\bar{\Q}}(E_j)	\qquad\text{and}\qquad	\End^{0}_{K'(j)}(E_0)=\End^{0}_{\bar{\Q}}(E_0).
	$$
	Since $j\in A(j_0)$, the elliptic curves $E_j$ and $E_0$ have CM by the same imaginary quadratic field, and hence are isogenous over $\bar{\Q}$. By \cite[Lemma~6.1]{masserwustholzisogenies}, there exists an extension $L(j)/K'(j)$ with $[L(j):K'(j)]\leq 12$ over which an isogeny
	$$
	\phi:E_j\longrightarrow E_0
	$$
	is defined.
	
	By \cite[Theorem~1.1]{pazuki}, there exist effectively computable constants $c_1,c_2>0$, depending only on $j_0$, such that
	\begin{equation}
		h(j)\leq c_1+c_2\log(\deg(\phi)).
	\end{equation}
	Applying \cite[Theorem~1.4]{gaudronremondpolarisations} over $L(j)$, we may choose $\phi$ so that
	$$
	\deg(\phi)\leq \kappa(E_0\times_KL(j)),
	$$
	where $\kappa(A)$ denotes the quantity defined in the introduction of \cite{gaudronremondpolarisations}. By its definition,
	\begin{align*}
		\log(\kappa(E_0\times_KL(j)))&\leq c_1^{\prime}+c_2^{\prime}\log([L(j):\Q])+\\
		&+c_3^{\prime}\log\!\left(\max\{h_F(E_0\times_KL(j)),\log([L(j):\Q]),1\}\right).
	\end{align*}
	
	Since $\log x\leq x$ for $x>0$, we have
	\begin{align*}
		\log(\max\{h_F(E_0\times_KL(j)),\log([L(j):\Q]),1\})\leq \\
	    \leq \max\{h_F(E_0\times_KL(j)),1\}+\log([L(j):\Q]).
	\end{align*}
	As $E_0$ has everywhere good reduction over $K$, it is semistable; invariance of the stable Faltings height under finite base change therefore gives
	$$
	h_F(E_0\times_KL(j))=h_F(E_0).
	$$
	
	The degree $[L(j):K(j)]$ is bounded independently of $j$, while
	$$
	[K(j):\Q]\leq [K:\Q][\Q(j):\Q].
	$$
	Hence $\log[L(j):\Q]\ll_K 1+\log[\Q(j):\Q]$. The desired bound follows.
\end{proof}

Following the above, we restrict attention below to singular moduli with CM field distinct from that of $j_0$.
   \subsubsection{Singular moduli with distinct CM fields}

We now specialize to the $j$-family of elliptic curves over $S:=\A^1\setminus\{0,1728\}$. Fix a singular modulus $j_0\in S$, and let $E_0$ be the corresponding CM elliptic curve with CM field $F_0$.

Before stating the height bound, we introduce some notation.

\begin{definition}\label{defnproblematic}
	Let $j$ be a singular modulus whose corresponding elliptic curve has CM by a field $F_j\neq F_0$. Set
	\begin{equation}
	\begin{aligned}
		\CP(j,j_0):=\{v\in\Sigma_{\Q(j,j_0),f}:&\,v(j-j_0)>0,\\
		&E_0\text{ has potentially supersingular reduction at }v\}.
	\end{aligned}
    \end{equation}
	We refer to $\CP(j,j_0)$ as the set of \textbf{supersingular places of proximity} of $j$ to $j_0$.
\end{definition}

The G-functions method yields the following height bound.
\begin{theorem}\label{heightboundcm}
	Let $j_0:=54000$, and write $\CP(j):=\CP(j,54000)$. There exist effectively computable constants $c_1,c_2>0$ such that
	\begin{equation}\label{eq:htboundmain}
		h(j)\leq c_1\big([\Q(j):\Q]+|\CP(j)|\big)^{c_2}
	\end{equation}
	for every singular modulus $j$ with CM field distinct from $\Q(\sqrt{-3})$.
\end{theorem}
\begin{proof} 	
	Let $f:\CE\rightarrow S$, with $S:=\mathbb{A}^1\backslash\{0,1728\}$, be the $j$-family of elliptic curves; see, for example, \cite[p.~47]{silvermanell}. Set $s_0=j_0:=54000\in S(\Q)$. Then $E_0:=\CE_{s_0}$ has CM by $\Q(\sqrt{-3})$.\footnote{See \cite[\href{https://www.lmfdb.org/EllipticCurve/2.0.3.1/144.1/CMa/2}{Elliptic Curve 144.1-CMa2}]{lmfdb}.}
	
	We take $x:=j-j_0$ as a local parameter at $s_0$. In this case no finite cover is needed to obtain the required local parameter. After the finite base-field extensions described in \Cref{section:backgroundgfuns}, the pair $(f,s_0)$ determines a single $2\times 2$ matrix $Y_G$ of G-functions.
	
	Fix $j=s\in S(\bar{\Q})$ such that $\CE_s$ has CM and
	$$
	\End^{0}_{\bar{\Q}}(\CE_s)\neq\Q(\sqrt{-3}).
	$$
	Let $\Sigma(s)$ denote the set of places $v$ of $L_s$ for which $s$ is $v$-adically close to $s_0$.
	
	If $\Sigma(s)=\emptyset$, then $h(x(s))$ is bounded by a constant depending only on the family and $j_0$; see the proof of the main height bound in \cite{papasbigboi}. Since $x(s)=j-j_0$ and $j_0$ is fixed, $h(j)$ is bounded by a constant independent of $j$.
	
	Assume henceforth that $\Sigma(s)\neq\emptyset$, and write
	$$
	\Sigma(s)_{\arch}:=\{v\in\Sigma(s):v\mid\infty\}\text{, } \Sigma(s)_{\ord}:=\{v\in\Sigma(s):v\text{ is ordinary for }E_0\},
	$$
	and
	$$
	\Sigma(s)_{\ssing}:=\{v\in\Sigma(s):v\text{ is supersingular for }E_0\}.
	$$
	
	Since $\End^{0}_{\bar{\Q}}(\CE_s)\neq \Q(\sqrt{-3})$, we have $\Sigma(s)_{\ord}=\emptyset$; see the proof of \Cref{thmordrels}. By \Cref{propandresupercm}, there exists
	$$
	R_{s,\ssing}\in L_s[X_{i,j}:1\leq i,j\leq 2]
	$$
	such that $\iota_v(R_{s,\ssing}(Y_G(x(s))))=0$ for all $v\in\Sigma(s)_{\ssing}$. Similarly, by \Cref{propbeukersarch}, there exists a polynomial 
	$$
	R_{s,\arch}\in L_s[X_{i,j}:1\leq i,j\leq 2]
	$$ 
	such that $\iota_v(R_{s,\arch}(Y_G(x(s))))=0$ for all $v\in\Sigma(s)_{\arch}$.
	
	Set
	$$
	R_s:=R_{s,\arch}R_{s,\ssing}.
	$$
	By construction, $\iota_v(R_s(Y_G(x(s))))=0$ for all $v\in\Sigma(s)$. Moreover, by \Cref{propandresupercm} and \Cref{propbeukersarch}, we have $R_s\notin I(\SL_2)$. Thus $R_s$ defines a global non-trivial relation in the sense of \cite[Chapter~VII.5]{andre1989g}. Since the auxiliary base-field extensions have bounded degree, depending only on $E_0$, the number of supersingular places occurring above $j$ is bounded by a constant multiple of $|\CP(j)|$. The height bound follows from Andr\'e--Bombieri's Hasse principle and the degree bounds of \Cref{propandresupercm} and \Cref{propbeukersarch}.
\end{proof}
\begin{remarks}
	1. The choice $j_0=54000$ is made only for convenience. The same argument applies to any singular modulus $j_0\neq 0,1728$, with $\Q(\sqrt{-3})$ replaced by the CM field of the corresponding elliptic curve. Changing $j_0$ alters the constants in \eqref{eq:htboundmain} through invariants related to the fiber $\CE_{s_0}$.
	
	2. The proof of \Cref{htboundgfunsgeneral} follows the same strategy as in \cite{papaspadicpart1}, with the local relations constructed in \Cref{section:cmpadicrels} as input and with appropriate changes in the description of the trivial relations.
\end{remarks}
   \subsection{The effective Brauer--Siegel problem}\label{section:brauersiegel}

Andr\'e points out in \cite[Remark~3, p.~201]{andre1989g} a connection between the G-functions method and the problem of obtaining an effective version of Siegel's class number bound. We recall Siegel's theorem in the following form.

\begin{theorem}[Siegel, \cite{siegelog}]\label{siegelogthm}
	Given $\epsilon>0$, there exists a constant $c(\epsilon)>0$ such that
	\begin{equation}\label{eq:siegeleq}
		h_K\geq c(\epsilon)|\disc(K)|^{\frac{1}{2}-\epsilon}
	\end{equation}
	for every imaginary quadratic number field $K$, where $h_K$ denotes the class number of $K$.
\end{theorem}

For a singular modulus $j$, let $E_j$ be the corresponding CM elliptic curve and set
$$
D(j):=\disc(\End_{\bar{\Q}}(E_j)).
$$
An effective version of \Cref{siegelogthm} is equivalent\footnote{See the introduction of \cite{papaseffbrsieg}.} to the existence, for every $\epsilon>0$, of an effective constant $c(\epsilon)>0$ such that
\begin{equation}\label{eq:siegelrestated}
	[\Q(j):\Q]\geq c(\epsilon)|D(j)|^{\frac{1}{2}-\epsilon}.
\end{equation}

Applying the proof of \cite[Theorem~1.5]{papaseffbrsieg}, with \Cref{heightboundcm} in place of \cite[Theorem~2.17]{papaseffbrsieg}, gives effectively computable constants $c_4,c_5>0$ such that
\begin{equation}\label{eq:brauersiegel1}
	|D(j)|\leq c_4\big([\Q(j):\Q]+|\CP(j)|\big)^{c_5}
\end{equation}
for every singular modulus $j$ with $\End^{0}_{\bar{\Q}}(E_j)\neq\Q(\sqrt{-3})$; see in particular equation~$(52)$ of \cite{papaseffbrsieg}.

For singular moduli with $\End^{0}_{\bar{\Q}}(E_j)=\Q(\sqrt{-3})$, combining \Cref{propisogenouscm} with the same Masser--W\"ustholz estimates used in \cite{papaseffbrsieg} gives an effective lower bound of the form \eqref{eq:siegelrestated}, with an exponent slightly smaller than $\frac{1}{2}$.

Landau proved the weaker, but still ineffective, estimate
$$
h_K\geq c|\disc(K)|^{\frac18}
$$
for imaginary quadratic fields $K$; see \cite{landau}. It is thus natural to ask first whether the contribution of $|\CP(j)|$ in \eqref{eq:brauersiegel1} can be controlled well enough to obtain an effective lower bound of Landau type, namely with some explicit positive exponent.

Accordingly, it would suffice to find effectively computable constants $C_1$, $C_2$, $C_3>0$ such that
\begin{equation}\label{eq:conjecturalbound}
	|\CP(j)|\leq C_1[\Q(j):\Q]^{C_2}|D(j)|^{\frac{1}{C_3}}
\end{equation}
for every singular modulus $j$ with $\End^{0}_{\bar{\Q}}(E_j)\neq\Q(\sqrt{-3})$, with $C_3>c_5$.

We now determine the exponent $c_5$ in \eqref{eq:brauersiegel1}, which is the product of the exponents arising from the Andr\'e--Bombieri Hasse principle and the Masser--W\"ustholz endomorphism estimates \cite{mwendoesti}.

For elliptic curves, the exponent in Masser--W\"ustholz's estimate may be taken to be $2+\epsilon$ for any $\epsilon>0$; see \cite{mwendoesti}. In the present setting, Andr\'e--Bombieri's Hasse principle gives exponent $9$ for non-trivial relations. By \Cref{trivialrelations} and the standard Picard--Vessiot interpretation of the Zariski closure of a fundamental solution matrix, the differential Galois group is $\SL_2$, and hence geometrically connected. Therefore, \cite[Remark~2, p.~138]{andre1989g} shows that every non-trivial relation is strongly non-trivial. We may thus use the stronger form of \cite[Chapter~VII, Theorem~5.2]{andre1989g}, which gives exponent $4$. Consequently, we may take $c_5=8+4\epsilon$ for any $\epsilon>0$. Hence any
$$
C_3>8
$$
suffices in \eqref{eq:conjecturalbound}, upon choosing $\epsilon>0$ sufficiently small.

\begin{remark}\label{remark Beukers bound}
	The term $|\CP(j)|$ in \Cref{heightboundcm} arises from the dependence on the place $v$ of the local relations constructed in \Cref{propandresupercm}. Instead of bounding the number of relevant places, one could try to bound directly the number of distinct local factors $R_{s,v}$ occurring as $v$ varies. For the application above, it would suffice to control this number by a quantity of the same form as the right-hand side of \eqref{eq:conjecturalbound}.
	
	In \cite[Theorem~5.4]{beukers}, Beukers obtains quantitative control over the algebraic parameter governing the dependence of $R_{s,v}$ on $v$. Combined with the Andr\'e--Bombieri Hasse principle, this yields a height bound of the form
	\begin{equation}\label{eq:badheightbound}
		h(j)\leq c_0[\Q(j):\Q]^4|D(j)|^2.
	\end{equation}
	The resulting estimate is insufficient. Indeed, after applying the Masser--W\"ustholz endomorphism estimates, a positive power of $|D(j)|$ remains on the right-hand side, so no effective lower bound for $[\Q(j):\Q]$ in terms of $|D(j)|$ follows.
\end{remark}
   \subsection{$p$-adic proximity of CM points}\label{section:grosszagier}

We use Lauter--Viray's generalization \cite{lauterviray} of the Gross--Zagier formula for singular moduli \cite{grosszagier} to study the $p$-adic proximity of CM points. This connection is also noted by Andr\'e in \cite[\S~3.3.4]{andresurvey}.

\subsubsection{The Gross--Zagier--Lauter--Viray formula}\label{section:lauterviraysummary}

Let $d_1,d_2$ be distinct discriminants of imaginary quadratic orders $\CO_1,\CO_2$, and define
\begin{equation}
	J(d_1,d_2):=\underset{\underset{\disc(\tau_j)=d_j}{[\tau_1],[\tau_2]}}{\prod}(j(\tau_1)-j(\tau_2))^{\frac{8}{w_1w_2}},
\end{equation}
where $[\tau_j]$ ranges over the elements of $\mathcal{H}/\SL_2(\Z)$ of discriminant $d_j$, and $w_j$ denotes the number of roots of unity in $\CO_j$.

Gross and Zagier established a formula for $J(d_1,d_2)$ when $d_1$ and $d_2$ are relatively prime fundamental discriminants \cite{grosszagier}. Lauter and Viray extended this formula to arbitrary distinct discriminants; see \cite[Theorem~1.1]{lauterviray}.

Gross and Zagier also described the primes dividing $J(d_1,d_2)$ under their hypotheses; see \cite[Corollary~1.6]{grosszagier}. Lauter and Viray obtained the corresponding result for arbitrary distinct discriminants:
\begin{theorem}[Lauter--Viray,{ \cite[Corollary 1.3]{lauterviray}}]\label{lauterviraytheorem}
	Let $d_1\neq d_2$ be discriminants as above and let $l$ be a prime with $l\mid J(d_1,d_2)$. Then
	\begin{enumerate}
		\item $d_2=d_1\cdot l^{2k}$ for some $k\in\Z\setminus\{0\}$, or
		\item there exists an integer $m=\frac{d_1d_2-x^2}{4}>0$ with $l\mid m$ such that the Hilbert symbol $(d_j,-m)_p\neq 1$ if and only if $p=l$.
	\end{enumerate}
	
	In particular, if $l\nmid d_1d_2$, then
	$$
	\left(\frac{d_1}{l}\right)=\left(\frac{d_2}{l}\right)=-1.
	$$
\end{theorem}
\subsubsection{Application to the set $\CP(j)$}\label{section:countingproximity}

We retain the notation of \Cref{heightboundcm} and set
$$
\CP_{\Q}(j):=\{p\text{ prime}:\exists\,v\in\CP(j)\text{ with }v\mid p\}.
$$
Then
$$
|\CP(j)|\leq[\Q(j):\Q]|\CP_{\Q}(j)|.
$$
Hence, in view of \eqref{eq:conjecturalbound}, it suffices to bound $|\CP_{\Q}(j)|$.

For $j_0=54000$ we have $D(j_0)=-12$. If $p\in\CP_{\Q}(j)$, then $p\mid J(D(j),-12)$. Since $F_j\neq\Q(\sqrt{-3})$, the first alternative in \Cref{lauterviraytheorem} cannot occur. The primes $p\mid 12D(j)$ contribute $O_{\epsilon}(|D(j)|^\epsilon)$. For each remaining $p\in\CP_{\Q}(j)$, \Cref{lauterviraytheorem} gives an integer $z_p$ with
$$
m_p=-3D(j)-z_p^2>0
$$
such that $p\mid m_p$ and
$$
(D(j),-m_p)_\ell=(-12,-m_p)_\ell=-1
$$
if and only if $\ell=p$.

Since $m_p>0$, counting the possible integers $z_p$ gives the trivial estimate
$$
|\CP_{\Q}(j)|\ll |D(j)|^{1/2}.
$$
It is natural to ask whether the Hilbert-symbol conditions above yield a bound
$$
|\CP_{\Q}(j)|\ll [\Q(j):\Q]^A|D(j)|^{1/C}
$$
for some $A>0$ and $C>8$. Such a bound would imply \eqref{eq:conjecturalbound}.

\subsubsection*{Acknowledgments}
The author thanks Yves Andr\'e for answering some questions about his work, as well as for pointing him to \cite{grosszagier}, which greatly influenced the direction of the exposition here. The author also thanks Chris Daw for his encouragement and many helpful discussions on the G-functions method. The author finally thanks Francesco Saettone for helpful comments that improved the readability of the exposition.

Work on this project started when the author was supported by Michael Temkin's ERC Consolidator Grant 770922 - BirNonArchGeom. For most of this work, the author received funding from the European Union (ERC, SharpOS, 101087910) and from the Israel Science Foundation (grant No. 2067/23). In the final stages of this work, the author was supported by the Minerva Research Foundation Member Fund while in residence at the Institute for Advanced Study for the academic year 2025--26.

\subsubsection*{Declaration of AI usage}
While revising this manuscript from an earlier version available on arXiv, the author used ChatGPT (OpenAI) for editorial and language assistance, principally to improve the exposition. The mathematical ideas, results, proofs, and arguments are solely due to the author.

\bibliographystyle{alpha}
\bibliography{biblio}
\end{document}